\documentclass[a4paper,12pt,reqno]{amsart}
\pdfoutput=1
\usepackage{amssymb}
\usepackage{amsmath}
\usepackage{mathtools}
\usepackage{enumitem}
\usepackage{mathrsfs}
\usepackage[all]{xy}
\usepackage{mathtools}
\usepackage{url}
\usepackage{longtable}
\usepackage[utf8]{inputenc}
\usepackage{bbm}
\usepackage{longtable}
\usepackage{siunitx}
\usepackage[OT2,T1]{fontenc}
\DeclareSymbolFont{cyrletters}{OT2}{wncyr}{m}{n}
\usepackage[british]{babel}

\usepackage[margin=1.3in]{geometry}
\numberwithin{equation}{section} \numberwithin{figure}{section}

\DeclareMathOperator{\Gal}{Gal}

\usepackage{siunitx}
\DeclareMathSymbol{\Sha}{\mathalpha}{cyrletters}{"58}
\usepackage{tikz-cd}
\usepackage{amsmath}
\usepackage{bm}
\newcommand{\code}[1]{\texttt{#1}}

\usepackage{listings}
\usepackage{color}

\newcommand*{\Comment}[1]{\hfill\makebox[3.2cm][l]{#1}}%
\newcommand*{\Commentt}[1]{\hfill\makebox[7cm][l]{#1}}%
\newcommand*{\Commenttt}[1]{\hfill\makebox[14cm][l]{#1}}%
\usepackage[toc,page]{appendix}

\newcommand\Z{\mathbb{Z}}

\newcommand\Q{\mathbb{Q}}

\newtheorem{lemma}{Lemma}

\newtheorem{thm}[lemma]{Theorem}
\newtheorem{prop}[lemma]{Proposition}

\theoremstyle{definition}

\newtheorem{defn}[lemma]{Definition}
\newtheorem{rem}[lemma]{Remark}

\usepackage{amsthm}

\newtheorem*{theorem*}{Theorem}
\usepackage[OT2,T1]{fontenc}
\DeclareSymbolFont{cyrletters}{OT2}{wncyr}{m}{n}
\DeclareMathSymbol{\Sha}{\mathalpha}{cyrletters}{"58}
\usepackage[OT2,T1]{fontenc}
\DeclareSymbolFont{cyrletters}{OT2}{wncyr}{m}{n}
\DeclareMathSymbol{\Sha}{\mathalpha}{cyrletters}{"58}

\usepackage[hyperindex,breaklinks]{hyperref}

\numberwithin{lemma}{section}

\begin{document}

\title[The Hasse norm principle for $A_n$-extensions]
{The Hasse norm principle for $A_n$-extensions}

\author{\sc André Macedo}

\address{André Macedo\\
Department of Mathematics and Statistics\\ University of Reading\\
Whiteknights, PO Box 220\\
Reading RG6 6AX\\
UK}
   \email{c.a.v.macedo@pgr.reading.ac.uk}
\urladdr{https://sites.google.com/view/andre-macedo}


\begin{abstract}
	We prove that, for every $n \geq 5$, the Hasse norm principle holds for a degree $n$ extension $K/k$ of number fields with normal closure $F$ such that $\operatorname{Gal}(F/k) \cong A_n$. We also show the validity of weak approximation for the associated norm one tori.

\end{abstract}
\maketitle

\section{Introduction}

Let $K/k$ be an extension of number fields with associated idèle groups $\mathbb{A}^{*}_{K}$ and $\mathbb{A}^{*}_{k}$ and let $\operatorname{N}_{K/k} : \mathbb{A}^{*}_{K} \to \mathbb{A}^{*}_{k}$ be the norm map on the idèles. We can view $K^*$ (respectively, $k^*$) as sitting inside $\mathbb{A}^{*}_{K}$ (respectively, $\mathbb{A}^{*}_{k}$) via the diagonal embedding and $\operatorname{N}_{K/k}$ naturally extends the usual norm map of the extension $K/k$. We say that the \textit{Hasse norm principle} (often abbreviated to HNP) holds for $K/k$ if the knot group

\medskip
\begin{center}
    $\mathfrak{K}(K/k):= (k^* \cap \operatorname{N}_{K/k}(\mathbb{A}^{*}_{K})) / \operatorname{N}_{K/k}(K^*)$
\end{center}
\medskip

\noindent is trivial, i.e. if every nonzero element of $k$ which is a local norm everywhere is a global norm.

The first example of the validity of this principle was established in 1931 by Hasse, who proved that the knot group $\mathfrak{K}(K/k)$ is trivial if $K/k$ is a cyclic extension (the Hasse norm theorem). Since then, much work has been done in the abelian case (see, for instance, \cite{frei}, \cite{gerth} or \cite{horie}), but results for the non-abelian and non-Galois cases are still limited. For example, if $F$ denotes the normal closure of $K/k$, it is known that the HNP holds for $K/k$ when

\begin{itemize}
    \item $[K:k]$ is prime (\cite{bar1});
    \item $[K:k]=n$ and $\operatorname{Gal}(F/k) \cong D_n$, the dihedral group of order $2n$ (\cite{bar2});
    \item $[K:k]=n$ and $\operatorname{Gal}(F/k) \cong S_n$, the symmetric group on $n$ letters (\cite{kun1}).

\end{itemize}
\smallskip

In this paper, we study the HNP for a degree $n$ extension $K/k$ with normal closure $F$ such that $\operatorname{Gal}(F/k)$ is isomorphic to $A_n$, the alternating group on $n$ letters.  We also look at \textit{weak approximation} - recall that this property is said to hold for a variety $X /k$ if $X(k)$ is dense (for the product topology) in $\prod\limits_{v} X(k_v)$, where the product is taken over all places $v$ of $k$ and $k_v$ denotes the completion of $k$ with respect to $v$. In particular, we examine weak approximation for the norm one torus $T=R^{1}_{K/k} \mathbb{G}_m$ associated to a degree $n$ extension of number fields $K/k$ with $A_n$-normal closure. 

The first non-trivial case is $n=3$. In this case, $K=F$ is a Galois extension of $k$ and the Hasse norm theorem tells us that the HNP holds for $K/k$. Moreover, using a result of Voskresenski\u{\i}, one can show that weak approximation holds for the associated norm one torus. In \cite{kun2} Kunyavski\u{\i} solved the case $n=4$ by showing that, for a quartic extension $K/k$ with $A_4$-normal closure, $\mathfrak{K}(K/k)$ is either $0$ or $\Z/2$ and both cases can occur. Additionally, he proved that the HNP holds for $K/k$ if and only if weak approximation fails for $T$. In this paper, we use several cohomological results about $A_n$-modules to prove the following theorem.

\begin{thm}\label{thman}
Let $n \geq 5$ be an integer. Let $K/k$ be a degree $n$ extension of number fields and let $F$ be its normal closure. If $\Gal(F/k) \cong A_n$, then the Hasse norm principle holds for $K/k$ and weak approximation holds for the norm one torus $T=R^{1}_{K/k} \mathbb{G}_m$.
\end{thm}

The layout of this paper is as follows. In Section 2, we use various cohomological and group-theoretic tools to establish the injectivity of an important restriction map on the cohomology of $A_n$. In Section 3, we look at the consequences of these results in the arithmetic of number fields. In particular, combining the results of Section 2 with the work of Colliot-Thélène and Sansuc on flasque resolutions and a theorem of Voskresenski\u{\i}, we prove Theorem \ref{thman} for $n \neq 6$. In Section 4, we exploit a computational method developed by Hoshi and Yamasaki to solve the remaining case $n=6$. All of the code used in Section 4 is provided in the Appendix. 

\medskip

\subsection*{Notation}
Throughout this paper, we fix the following notation.

\bigskip

\begin{longtable}{p{1cm} p{13cm}}
$k$ & a number field; \\
$\overline{k}$ & an algebraic closure of $k$; \\
$v$ & a place of $k$; \\
$k_v$ & the completion of $k$ at $v$.
\end{longtable}

\medskip
For a variety $X$ over a field $K$, we use the notation

\bigskip
\begin{longtable}{p{3cm} p{11.5cm}}
$X_L=X \times_{K} L$ & the base change of $X$ to a field extension $L/K$; \\
$\overline{X}=X \times_{K} \overline{K}$ & the base change of $X$ to an algebraic closure of $K$; \\
$\operatorname{Pic}X$ & the Picard group of $X$.
\end{longtable}

\vspace{5pt}
We define $\mathbb{G}_{m,K}=\operatorname{Spec}(K[t,t^{-1}])$ to be the multiplicative group over a field $K$ and, if $K$ is apparent from the context, we omit it from the subscript and simply write $\mathbb{G}_{m}$. Given a $K$-torus $T$, we write $\hat{T}$ for its character group $\operatorname{Hom}(\overline{T},\mathbb{G}_{m,\overline{K}})$. If $L/K$ is a finite extension of fields and $T$ is an $L$-torus, we denote the Weil restriction of $T$ from $L$ to $K$ by $ R_{L/K} T$. We use the notation $R^{1}_{L/K} \mathbb{G}_m$ for the norm one torus, defined as the kernel of the norm map $ \operatorname{N}_{L/K} : R_{L/K}\mathbb{G}_m \to \mathbb{G}_{m}$.

Let $G$ be a finite group. The label `$G$-module' shall always mean a free $\Z$-module of finite rank equipped with a right action of $G$. For a $G$-module $A$ and $q \in \Z$, we denote the Tate cohomology groups by $\hat{\operatorname{H}}^q(G,A)$ and the kernel of the restriction map $\hat{\operatorname{H}}^q(G,A) \xrightarrow[]{\operatorname{Res}} \prod\limits_{g \in G} \hat{\operatorname{H}}^q(\langle g \rangle,A)$ by $\Sha_{\omega}^q(G,A)$. Since $\hat{\operatorname{H}}^q(G,A)=\operatorname{H}^q(G,A)$ for $q \geq 1$, we will omit the hat in this case. We also use the notation $Z(G)$, $[G,G]$, $G^{\sim}$ and $M(G)$ for the center, the derived subgroup, the dual group $\operatorname{Hom}(G, \Q/\Z)$ and the Schur multiplier $\hat{\operatorname{H}}^{-3}(G,\Z)$ of $G$, respectively. Given elements $g,h \in G$, we use the conventions $[g,h]=gh g^{-1} h^{-1}$ and $g^h=h g h^{-1}$.

\pagebreak
\subsection*{Acknowledgements.} I would like to thank my supervisor Rachel Newton for suggesting this problem and for helpful comments on an earlier version of this manuscript. I am also grateful to Prof. Boris Kunyavski\u{\i} for bringing up the importance of Lemma \ref{lemcor}. This work was supported by the FCT doctoral scholarship SFRH/BD/117955/2016.

\section{Group cohomology of $A_n$-modules}
The goal of this section is to establish several cohomological facts about $A_n$-modules. We start by stating some useful group-theoretic facts.

\begin{rem}\label{shmult}
Recall that, for $n \geq 5$, $A_n$ is a non-abelian simple group and hence perfect. Moreover, its Schur multiplier $M(A_n)=\hat{\operatorname{H}}^{-3}(A_{n},\mathbb{Z})$ is given as follows (see Theorem 2.11 of \cite{HH92}):

\[ M(A_n) = \begin{cases*}
                    0  & if $n \leq 3$ \\
                    \Z / 2  & if $n \in \{4,5\}$ or $n \geq 8$ \\
                    \Z / 6 & if  $n \in \{ 6,7 \}$.
                 \end{cases*} \]

\end{rem}

Given a copy $H$ of $A_{n-1}$ inside $G=A_n$, we have a corestriction map on cohomology 

\begin{center}
    
$\operatorname{Cor}^{H}_{G}:M(H) \to M(G)$. 
\end{center}

\smallskip
This map will play an important role for us later, so we begin by establishing the following result.

\begin{lemma}\label{lemcor}
Let $n \geq 8$ and let $H$ be a copy of $A_{n-1}$ inside $G=A_n$. Then, the corestriction map $\operatorname{Cor}^{H}_{G}$ is surjective.
\end{lemma}

In order to prove this lemma, we will use multiple results about covering groups of $S_n$ and $A_n$ together with the characterization of the image of $\operatorname{Cor}^{H}_{G}$ given in Lemma 4 of \cite{DP87}. To put this plan into practice, we need the following concepts.

\begin{defn}
Let $G$ be a finite group. A \textit{stem extension} of $G$ is a group $\widetilde{G}$ containing a normal subgroup $K$ such that $\widetilde{G} / K \cong G$ and $K \subseteq Z(\widetilde{G}) \cap [\widetilde{G},\widetilde{G}]$. A \textit{Schur covering group} of $G$ is a stem extension of $G$ of maximal size.
\end{defn}

It is a well-known fact that a stem extension of a finite group $G$ always exists (see Theorem 2.1.4 of \cite{kar}). Additionally, the base normal subgroup $K$ of a Schur covering group $\widetilde{G}$ of $G$ coincides with its Schur multiplier $\hat{\operatorname{H}}^{-3}(G,\Z)$ (see Section 9.9 of \cite{gru}). In \cite{S11}, Schur completely classified the Schur covering groups of $S_n$ and $A_n$. He also gave an explicit presentation of a cover of $S_n$, as follows.

\begin{prop}\label{pres}
Let $n \geq 4$ and let $U$ be the group with generators $z,t_1,\cdots, t_{n-1}$ and relations
\begin{enumerate}
    \item $z^2=1$;
    \item $z t_i = t_i z $, \textnormal{for $1 \leq i \leq n-1$};
    \item $t_i^2=z$, \textnormal{for $1 \leq i \leq n-1$};
    \item $(t_i t_{i+1})^3 =z $, \textnormal{for $1 \leq i \leq n-2$};
    \item $t_i t_j = z t_j t_i$, \textnormal{for $|i-j| \geq 2$ and $1 \leq i,j \leq n-1$}.
\end{enumerate}
\medskip

Then $U$ is a Schur covering group of $S_n$ with base normal subgroup $K = \langle z \rangle $. Moreover, if $\overline{t_i}$ denotes the transposition $(i \hspace{5pt} i+1)$ in $S_n$, then the map

\begin{align*}
  \pi \colon U &\longrightarrow S_n \\
  z &\longmapsto 1 \\
  t_i &\longmapsto \overline{t_i}
\end{align*}

\noindent is surjective and has kernel $K$.

\end{prop}

\begin{proof}
See Schur's original paper \cite{S11} or Chapter 2 of \cite{HH92} for a more modern exposition.
\end{proof}

\begin{rem}
An immediate consequence of this last proposition is that the Schur multiplier of $S_n$ is isomorphic to $\Z / 2$ for $n \geq 4$.
\end{rem}

Using the Schur cover of $S_n$ given in Proposition \ref{pres}, one can also construct a Schur covering group of $A_n$ for $n \geq 8$.

\begin{lemma}\label{schurcov}
In the notation of Proposition \ref{pres}, the group $V:=\pi^{-1}(A_n)$ defines a Schur covering group of $A_n$ for every $n \geq 8$.
\end{lemma}

\begin{proof}
It is well-known that $A_n$ is generated by the $n-2$ permutations $\overline{e_i}:=\overline{t_1} . \overline{t_{i+1}}=(1 \hspace{5pt} 2) (i+1 \hspace{5pt} i+2)$ for $1 \leq i \leq n-2$. Hence, $V=\pi^{-1}(A_n)$ is generated by $z,e_1,\cdots,e_{n-2}$, where $e_i:=t_1 t_{i+1}$ for $1 \leq i \leq n-2$. Clearly, we have $K \subseteq Z(V)$ and $V / K \cong A_n$. As the Schur multiplier of $A_n$ is also $\Z /2$ for $n \geq 8$, in order to show that $V$ defines a Schur covering group of $A_n$ it suffices to prove that $K \subseteq [V,V]$.

\medskip

\noindent \textbf{Claim:} $z=[e_1^{-1} e_2 e_1 , e_2]$.

\noindent \textbf{Proof of claim:} This follows from a standard computation using the identities $(e_1 e_2)^3=z$, $e_1^3=z$ and $e_i^2=z$ for $2 \leq i \leq n-2$, which follow directly from the relations satisfied by the $t_i$.

\medskip
Given the claim, it follows that $K = \langle z \rangle$ is contained in $[V,V]$, as desired.

\end{proof}

Given a copy $H$ of $A_{n-1}$ inside $A_n$, one can subsequently repeat the same procedure of this last lemma and further restrict $\pi$ to $W:=\pi^{-1}(H)$ to seek a Schur covering group of $H$. The same argument works, but with two small caveats.

First, it is necessary to assure that we still have $z \in [W,W]$. This is indeed the case since, for $n \geq 7$, any subgroup $H \leq A_n$ isomorphic to $A_{n-1}$ is conjugate to the point stabilizer $(A_{n})_n$ of the letter $n$ in $A_n$ (this is a consequence of Lemma 2.2 of \cite{wil}). Therefore, we have $H= {(A_{n})_n}^{\pi(x)}$ for some $x \in U$ and hence $z=z^{x}=[e_1^{-1} e_2 e_1,e_2]^x=[(e_1^{-1} e_2 e_1)^x,e_2^x] $ is in $ [W,W]$, as clearly $\overline{e_1},\overline{e_2} \in (A_{n})_n $. 

Second, note that we are making use of the fact that the Schur multipliers of $A_{n-1}$ and $S_n$ coincide, which is only true for $n \geq 9$ (recall that $M(A_7)=\Z/6$). However, it is still true that $\pi^{-1}(A_{7})$ gives a (non-maximal) stem extension of $A_7$ by the same reasoning as above. We have thus established the following result.

\begin{lemma}\label{schurcov2}
Let $n \geq 8$ and let $H$ be a copy of $A_{n-1}$ inside $A_n$. Then, the restriction to $W=\pi^{-1}(H)$ of the Schur cover $V$ of $A_n$ given in Lemma \ref{schurcov} defines a stem extension of $H$. 
\end{lemma}

We can now prove Lemma \ref{lemcor}.

\begin{proof}[Proof of Lemma \ref{lemcor}]

Let $V$ be the Schur covering group of $G$ constructed in Lemma \ref{schurcov}. We then have a central extension

\begin{equation*}
    1 \to M(G) \to V \xrightarrow[]{\pi} G \to 1,
\end{equation*}

\medskip

\noindent where we identified the base normal subgroup $K$ of $V$ with the Schur multiplier $M(G)$ of $G$. Since $M(G) \subset [V,V]$ by the definition of a Schur cover, $V$ is a \textit{generalized representation group} of $G$, as defined on p. 310 of \cite{DP87}. Therefore, by Lemma 4 of \cite{DP87}, we have an isomorphism $\operatorname{Cor}^{H}_{G}(M(H)) \cong M(G) \cap [W,W]$, where $W=\pi^{-1}(H)$. Hence, it is enough to show that $M(G) \cap [W,W] = M(G)$. By Lemma \ref{schurcov2}, $W$ defines a stem extension of $H$ for $n \geq 8$, so that we immediately get $M(G) \subset [W,W]$. It follows that $M(G) \cap [W,W] = M(G)$, as desired.
\end{proof}

In order to proceed with our cohomological analysis, we need to recollect some group-theoretic objects. Let $H$ be a subgroup of a finite group $G$. Recall that we have the augmentation map $\epsilon : \Z[G/H] \to \Z $ defined by $\epsilon : Hg \mapsto 1$ for any $ Hg \in G/H$. This map produces the exact sequence of $G$-modules

\begin{equation}\label{eps}
    0 \to I_{G/H} \to \Z[G/H] \xrightarrow[]{\epsilon} \Z \to 0 ,
\end{equation}

\medskip
\noindent where $I_{G/H}=\operatorname{ker}(\epsilon)$ is the augmentation ideal. Dually, we also have a map $\eta : \Z \to \Z[G/H] $ defined by $\eta : 1 \mapsto N_{G/H}$, where $N_{G/H}=\sum\limits_{Hg \in G/H} Hg$. This produces the exact sequence of $G$-modules

\begin{equation}\label{eta}
    0 \to \Z \xrightarrow[]{\eta} \Z[G/H] \to J_{G/H}  \to 0 ,
\end{equation}

\medskip
\noindent where $J_{G/H}=\operatorname{coker}(\eta)$ (called the Chevalley module of $G/H$) is the dual module $\operatorname{Hom}(I_{G/H},\Z)$ of $I_{G/H}$.

\medskip

For any $g \in G$, we can consider the restriction maps 

\begin{equation*}
    \operatorname{Res}_{g}: \operatorname{H^2}(G,J_{G/H}) \to \operatorname{H}^2(\langle g \rangle, J_{G/H})
\end{equation*}

\medskip
\noindent and aggregate all of these functions together in order to get a homomorphism of $G$-modules

\begin{equation*}
    \operatorname{Res}: \operatorname{H^2}(G,J_{G/H}) \to \prod\limits_{g \in G} \operatorname{H}^2(\langle g \rangle, J_{G/H}).
\end{equation*}

It turns out that the kernel of this map (denoted by $\Sha^2_{\omega}(G,J_{G/H})$) is of extreme importance in the arithmetic of number fields, as we will see in the next section. We describe this kernel for our case of interest $G=A_n$, $H\cong A_{n-1}$ and $n \geq 8$ (the cases $n \leq 7$ will be treated separately).

\medskip
\begin{prop}\label{sha2}

Let $n \geq 8$ and let $H$ be a copy of $A_{n-1}$ inside $G=A_n$. Then, we have $\Sha^2_{\omega}(G,J_{G/H} )=0$.
\end{prop}

\begin{proof}
Taking the $G$-cohomology of the exact sequence \eqref{eta} gives the exact sequence of abelian groups

\begin{equation*}
 \operatorname{H}^2(G,\Z[G/H]) \to \operatorname{H}^2(G,J_{G/H}) \to \operatorname{H}^3(G,\Z) \xrightarrow[]{\overline{\eta}} \operatorname{H}^3(G,\Z[G/H]),
\end{equation*}
 
 \medskip
 \noindent where $\overline{\eta}$ is the map induced on the degree $3$ cohomology groups by the norm map $\eta$. Applying Shapiro's lemma and using the fundamental duality theorem in the cohomology of finite groups (see, for example, Section VI.7 of \cite{brown}), we have $\operatorname{H}^2(G,\Z[G/H]) \cong \operatorname{H}^2(H,\Z) \cong \hat{\operatorname{H}}^{-2}(H,\Z) \cong H/[H,H] = 0$, as $H$ is perfect. Therefore, this last exact sequence becomes

\begin{equation*}
0 \to \operatorname{H}^2(G,J_{G/H}) \to \operatorname{H}^3(G,\Z) \xrightarrow[]{\overline{\eta}} \operatorname{H}^3(G,\Z[G/H]),
\end{equation*}

\medskip

\noindent which shows that $\operatorname{H}^2(G,J_{G/H})=0$ if $\overline{\eta}$ is injective. Since the composition of the map $\overline{\eta}$ with the isomorphism in Shapiro's lemma 

\begin{center}
    $\operatorname{H}^3(G,\Z) \xrightarrow[]{\overline{\eta}} \operatorname{H}^3(G,\Z[G/H]) \xrightarrow[]{\cong} \operatorname{H}^3(H,\Z)$
\end{center}

\noindent gives the restriction map (see Example 1.27(b) of \cite{milne}), it is enough to prove that the restriction 

 \begin{center}
     $\operatorname{Res}^{G}_{H}: \operatorname{H}^3(G,\Z) \to \operatorname{H}^3(H,\Z)$
 \end{center} 
 
 \noindent is injective. Again, by the duality in the cohomology of finite groups, this is the same as proving that the corestriction map (dual to $\operatorname{Res}^{G}_{H}$)
 
 \begin{center}
     $\operatorname{Cor}^{H}_{G}: \hat{\operatorname{H}}^{-3}(H,\Z) \to \hat{\operatorname{H}}^{-3}(G,\Z)$
 \end{center} 
 
 \noindent is surjective. But this is the content of Lemma \ref{lemcor}, so it follows that $\operatorname{H}^2(G,J_{G/H})$ is trivial and therefore $\Sha_{\omega}^2(G,J_{G/H})=0$, as desired.

\end{proof}

\section{Arithmetic consequences}

\medskip

In this section, we delve into the consequences of the cohomological results of Section 2 in the arithmetic of number fields. In particular, we will recall how the group $ \Sha^2_{\omega}(G,J_{G/H})$ governs two important local-global principles, the Hasse norm principle and weak approximation. Specifying to the case $G=A_n$, $H \cong A_{n-1}$ and using Proposition \ref{sha2}, we will prove Theorem \ref{thman} for $n \geq 8$. The remaining cases ($n \leq 7$) will be solved using a result of Colliot-Thélène and Sansuc and a computational method adapted from work of Hoshi and Yamasaki. 

\medskip

Let $k$ be a number field and let $T$ be a $k$-torus. We introduce the defect to weak approximation for $T$

\begin{equation*}
A(T)=(\prod\limits_{v}T(k_v))/\overline{T(k)},
\end{equation*}
\noindent where the product is taken over all places $v$ of $k$ and $\overline{T(k)}$ denotes the closure (with respect to the product topology) of $T(k)$ in $\prod\limits_{v}T(k_v)$. We say that weak approximation holds for $T$ if and only if $A(T)=0$. 

We also define the Tate-Shafarevich group of $T$ as 

\begin{equation*}
\Sha(T)=\operatorname{ker}(\operatorname{H}^1(k,T) \to \prod\limits_{v}\operatorname{H}^1(k_v,T_{k_v})),
\end{equation*}

\noindent where the product runs over all places $v$ of $k$. It is known that this group controls the validity of the Hasse principle for every principal homogeneous space under $T$. In fact, the Hasse principle holds for every such space if and only if $\Sha(T)=0$. 

\smallskip

The following result remarkably connects weak approximation with the Hasse principle by combining the two groups $A(T)$ and $\Sha(T)$ in an exact sequence.

\begin{thm}[Voskresenski\u{\i}]\label{thmvosk}
Let $T$ be a torus defined over a number field $k$ and let $X/k$ be a smooth projective model of $T$. Then there exists an exact sequence

\begin{equation*}
    0 \to A(T) \to \operatorname{H}^1(k,\operatorname{Pic}\overline{X})^{\sim} \to \Sha(T) \to 0.
\end{equation*}

\end{thm}

\begin{proof}
See Theorem 6 of \cite{vosk}.
\end{proof}
\smallskip

Let us now specialize $T$ to be the norm one torus $R^{1}_{K/k} \mathbb{G}_m$ of an extension $K/k$ of number fields. In this case, we have $ \mathfrak{K}(K/k) \cong \Sha(T) $ (see p. 307 of \cite{platonov}). Therefore, the cohomology group $\operatorname{H}^1(k,\operatorname{Pic}\overline{X})$ in the previous theorem is pivotal in the study of the HNP for $K/k$ and weak approximation for $T$. A very useful tool to deal with this object is flasque resolutions, as introduced in the work of Colliot-Thélène and Sansuc. We recall here the main definitions and refer the reader to \cite{coll} and \cite{coll2} for more details on this topic.

\subsection*{Flasque resolutions}
Let $G$ be a finite group and let $A$ be a $G$-module. The module $A$ is said to be \textit{flasque} if $\hat{\operatorname{H}}^{-1}(G',A)=0$ for every subgroup $G'$ of $G$ and \textit{coflasque} if $\operatorname{H}^{1}(G',A)=0$ for every subgroup $G'$ of $G$. Moreover, $A$ is called a \textit{permutation} module if it admits a $\Z$-basis permuted by $G$ and an \textit{invertible} module if it is a direct summand of a permutation module. A \textit{flasque resolution} of $A$ is an exact sequence of $G$-modules

\begin{equation*}
    0 \to A \to P \to M \to 0
\end{equation*}

\medskip
\noindent where $P$ is a permutation module and $M$ is flasque. Dually, a \textit{coflasque resolution} of $A$ is an exact sequence of $G$-modules

\begin{equation*}
    0 \to N \to Q \to A \to 0
\end{equation*}

\medskip
\noindent where $Q$ is a permutation module and $N$ is coflasque. 

\medskip

It turns out that there is a very direct relation between the group $\operatorname{H}^1(k,\operatorname{Pic}\overline{X})$ and flasque resolutions of the $G$-module $\hat{T}$, as the following result shows.

\medskip
\begin{thm}[Colliot-Thélène \& Sansuc]\label{thmcs}
Let $T$ be a torus defined over a number field $k$ and split by a finite Galois extension $F/k$ with $G=\operatorname{Gal}(F/k)$. Suppose that

\begin{equation*}
    0 \to \hat{T} \to P \to M \to 0
\end{equation*}

\medskip
\noindent is a flasque resolution of the $G$-module $\hat{T}$ and let $X/k$ be a smooth projective model of $T$. Then, we have

\begin{equation*}
    \operatorname{H}^1(k,\operatorname{Pic}\overline{X}) = \operatorname{H}^1(G,\operatorname{Pic}X_F) = \operatorname{H}^1(G,M) .
\end{equation*}
\end{thm}

\begin{proof}
See Lemme 5 and Proposition 6 of \cite{coll}.
\end{proof}

\medskip

We proceed by presenting a very useful description of the group $\operatorname{H}^1(G,M) $ in the conclusion of the previous theorem.

\medskip
\begin{prop}\label{char}
$\operatorname{H}^1(G,M)=\Sha_{\omega}^2(G,\hat{T})$.
\end{prop}

\begin{proof}
See Proposition 9.5(ii) of \cite{coll2}.
\end{proof}

\medskip
Using this characterization, we can now prove Theorem \ref{thman} for $n \neq 6$ (the case $n=6$ will be treated separately in the next section).

\medskip
\begin{proof}[Proof of Theorem \ref{thman} for $n \neq 6$]
 Set $G=\Gal(F/k) \cong A_n$ and $H=\Gal(F/K)$. Observe that such a group $H$ is necessarily isomorphic to $A_{n-1}$, since it has index $n$ in $A_n$. We have two cases:
 
 \smallskip

\textbf{Case} $\bm{n \geq 8}$: By Theorems \ref{thmvosk} and \ref{thmcs} and Proposition \ref{char}, it is enough to establish that the group $ \Sha_{\omega}^2(G,\hat{T})$ is trivial, where $T=R^{1}_{K/k} \mathbb{G}_m$ is the norm one torus associated to the extension $K/k$. Moreover, it is a well-known fact that $\hat{T}=J_{G/H}$ as $G$-modules, so it is sufficient to prove that $ \Sha_{\omega}^2(G,J_{G/H})=0$. But this was shown in Proposition \ref{sha2} of Section 2, so the result follows.

\smallskip

\textbf{Cases} $\bm{n =5} \textbf{ and } \bm{n=7}$: Since $n$ is a prime number, these cases follow from a direct application of Proposition 9.1 of \cite{coll2}. In this proposition, the authors show that there exists a $k$-torus $T_1$ such that the variety $T \times_{k} T_1$ is $k$-rational, where $T=R^{1}_{K/k} \mathbb{G}_m$ is the norm one torus associated to $K/k$. This result is in its turn equivalent to the fact that any flasque module $M$ in a flasque resolution of $\hat{T}$ is invertible (see Proposition 9.5(i) of \cite{coll2}), which is a stronger property than being coflasque. Therefore, the group $\operatorname{H}^1(G,M)$ vanishes and so, by Theorem \ref{thmcs}, the middle group of Voskresenski\u{\i}'s exact sequence in Theorem \ref{thmvosk} is trivial. Hence, we conclude that $A(T)=0=\Sha(T)$, as desired.

\end{proof}

\section{The case $n=6$}

In this section, we finish the proof of Theorem \ref{thman} by using the computer algebra system GAP to establish the remaining case $n=6$. More precisely, we devise an algorithm that, given a finite group $G$ and a non-normal subgroup $H$ such that $\operatorname{Core}_G(H):=\bigcap\limits_{g \in G} g^{-1}Hg$ is trivial (for example, this is always the case if $G$ is simple), outputs the invariant $\operatorname{H}^1(G,M)$ of Theorem \ref{thmcs} for the norm one torus. We use two ingredients to achieve this: First, we construct a routine in GAP that computes the matrix representation of the action of $G$ on the Chevalley module $J_{G/H}$. Second, we make use of the GAP algorithms\footnote{The code for these algorithms is available on the web page \url{https://www.math.kyoto-u.ac.jp/~yamasaki/Algorithm/RatProbAlgTori/}.} developed by Hoshi and Yamasaki in \cite{hoshi} to construct flasque resolutions. Before we present our method, we need a few preliminaries.

\smallskip

\begin{defn}[Definition 1.26 of \cite{hoshi}]
Let $G$ be a finite subgroup of $\operatorname{GL}(n, \Z)$. The $G$-lattice $M_G$ is defined to be the $G$-lattice with a $\Z$-basis $\{u_1, \dots , u_n\}$ and right action of $G$ given by
$u_i.g = \sum\limits_{j=1}^{n}
 a_{i,j}u_j$, where $g = [a_{i,j} ]_{i,j=1}^{n} \in G$.
\end{defn}

In \cite{hoshi} the authors study the rationality of low-dimensional algebraic tori via the properties of the corresponding group modules, for which they create multiple algorithms. In particular, given a finite subgroup $G$ of $ \operatorname{GL}(n,\Z)$, they design the functions $\code{H1}$ and $\code{FlabbyResolution}$ (see Sections 5.0 and 5.1 of \cite{hoshi}, respectively) computing the cohomology group $\operatorname{H}^1(G,M_G)$ and producing a flasque resolution of the $G$-module $M_G$, respectively. For instance, by invoking the command 

\medskip

   \noindent $\code{gap> FlabbyResolution(G).actionF;}$

\medskip
\noindent in GAP, one can access the matrix representation of the action of $G$ on a flasque module in a flasque resolution of $M_G$.

\medskip
 Let $G$ be a finite group and $H$ a non-normal subgroup of $G$ with trivial normal core $\operatorname{Core}_G(H)$. Set $d=|G/H|$ and fix a set of right-coset representatives  ${L=\{ H g_1, \dots ,  H g_d\}}$ of $H$ in $G$. In this way, we have $\Z[G/H] = \sum\limits_{i=1}^{d}  H g_i \Z$ and $N_{G/H}=\sum\limits_{i=1}^{d} H g_i \in \Z[G/H]$.

\medskip
\begin{sloppypar}
Our first goal is to establish an isomorphism between the $G$-module $J_{G/H}$ and the $R_G$-module $M_{R_G}$, where ${R_G \leq \operatorname{GL}(d-1, \Z)}$ is a group (to be defined below) isomorphic to $G$. We accomplish this by using the representation of $G$ associated to its right action on $J_{G/H}$. More precisely, consider the $\Z$-basis 

\medskip
\begin{center}
    $B=\{ H g_1 + N_{G/H} \Z, \dots ,H g_{d-1} + N_{G/H} \Z\}$
\end{center} 
\medskip

\noindent of $J_{G/H}$. Since the submodule $ N_{G/H} \Z$ is fixed by the action of any element of $G$, we will omit it when working with elements of $B$. Given $g \in G$, we build a matrix $R_g \in \operatorname{GL}(d-1, \Z)$ as follows. 
\end{sloppypar}
 For any $H g_i  \in B$, we have $ (H g_i).g = Hg_{\sigma(i)}$ for some $1 \leq \sigma(i) \leq d$. There are two cases:

\bigskip
\textbf{1)} If $\sigma(i) < d$, then the $k$-th entry of the $i$-th row of $R_g$ is set to be equal to $1$ if $k=\sigma(i)$ and $0$ otherwise. 

\smallskip
\textbf{2)} If $\sigma(i) = d$, then the $k$-th entry of the $i$-th row of $R_g$ is set to be equal to $-1$ for every $k$.

\bigskip
Let $R_G$ be the group $\langle R_g \text{ } | \text{ } g \in G \rangle \leq \operatorname{GL}(d-1,\Z)$. It is easy to see that the map

\begin{align*}
  \rho_G   \colon & G  \longrightarrow  R_G \\
   & g \longmapsto  R_g
\end{align*}

\medskip
\noindent is the representation of $G$ corresponding to its action on $J_{G/H}$. Clearly we have $\ker \rho_G = \operatorname{Core}_G(H)$, which we are assuming is trivial. Hence, $\rho_G$ is faithful and thus it yields an isomorphism $G \cong R_G$. Moreover, identifying $R_g \in R_G$ with the corresponding element $g \in G$, it is straightforward to check that the map

\begin{align*}
  \psi  \colon & M_{R_G}  \longrightarrow  J_{G/H} \\
   & \sum\limits_{i=1}^{d-1} \lambda_i u_i \longmapsto  \sum\limits_{i=1}^{d-1} \lambda_i  H g_i+ N_{G/H}\Z
\end{align*}

\medskip

\noindent defines an isomorphism of group modules. 

\medskip
With the tools introduced so far, we are now able to construct the function $\code{FlasqCoho(G,H)}$ (presented in the Appendix) in GAP that computes the cohomology group $\operatorname{H}^1(G,M)$, where $M$ is a flasque module in a flasque resolution of the Chevalley module $J_{G/H}$. The necessary steps to assemble this function are the following.

\bigskip

\textbf{Step 1)} Fix a set $\code{gens}$ of generators of $G$ and construct the matrix group $R_G=\langle R_g \text{ } | \text{ } g \in \code{gens} \rangle$ using the following two functions
\smallskip

\begin{itemize}
    \item $\code{row(s,d)}$ (an auxiliary routine to \code{action}), returning the $i$-th row of the matrix $R_g$ as explained on page 9;
    \item $\code{action(G,H)}$, constructing the matrices $R_g$ for $g \in \code{gens}$ and returning the group $R_G$.

\end{itemize}

\smallskip
The code for these two functions is also provided in the Appendix. The group $R_G=\code{action(G,H)}$ is then a subgroup of $\operatorname{GL}(d-1,\Z)$ isomorphic to $G$ such that $M_{R_G} \cong J_{G/H}$.

\medskip

\textbf{Step 2)} Create a flasque resolution of the $R_G$-module $M_{R_{G}}$ and access its flasque module $M'$ using the commands

\bigskip
   \noindent $\code{gap> FR:=FlabbyResolution(RG);}$
   
  \noindent  $\code{gap> FM:=FR.actionF;}$

\bigskip

The object $\code{FM}$ is the matrix representation group of the action of $R_G$ on $M'$. Note that, by the inflation-restriction exact sequence, we have $\operatorname{H}^1(R_G,M')\cong \operatorname{H}^1(\code{FM},M_{\code{FM}})$.

\medskip

\textbf{Step 3)} Obtain the group $\operatorname{H}^1(\code{FM},M_{\code{FM}})\cong \operatorname{H}^1(G,M)$ using the function $\code{H1}$.

\bigskip

   \noindent $\code{gap> H1(FM);}$

\bigskip

The result of this line is the final output of the algorithm.

\bigskip

Using this computational method, we can now establish the remaining case of Theorem 1.1.

\begin{proof}[Proof of the case $n = 6$ in Theorem \ref{thman}]

Set $G=\Gal(F/k) \cong A_6$ and $H = \Gal(F/K)$. By Theorems \ref{thmvosk} and \ref{thmcs}, it is enough to prove that the cohomology group $\operatorname{H}^1(G,M)$ is trivial, where $M$ is a flasque module in a flasque resolution of the $G$-module $\hat{T}=J_{G/H}$ and $T=R^{1}_{K/k} \mathbb{G}_m$ is the norm one torus associated to the extension $K/k$. 

As in the case $n \neq 6$, we have $H \cong  A_5$. Notice that, up to conjugation, there are exactly two distinct subgroups of $A_6$ isomorphic to $A_5$, namely ${H_1=\langle (1 \hspace{5pt} 2 \hspace{5pt}3\hspace{5pt}4\hspace{5pt}5),(1\hspace{5pt}2\hspace{5pt}3) \rangle}$ and $H_2=\langle (1\hspace{5pt}2\hspace{5pt}3\hspace{5pt}4\hspace{5pt}5),(1\hspace{5pt}4)(5\hspace{5pt}6) \rangle$. Moreover, it suffices to check the vanishing of $\operatorname{H}^1(G,M)$ for one subgroup $H$ in each conjugacy class (this follows from the fact that two subgroups $H_1$ and $H_2$ are conjugate if and only if the two $G$-sets $G/H_1$ and $G/H_2$ are isomorphic). Using the above algorithm, we obtained $\operatorname{H}^1(G,M)=0$ in both cases, as desired.

\end{proof}

\begin{rem}
The computation used for the case $n=6$ in the previous proof can be reproduced for other small values of $n$. We have checked that for $n \leq 11$ the algorithm confirms our results, giving the trivial group for $n \neq 4$ and producing the counterexample $\operatorname{H}^1(A_4,M)=\Z / 2$ for $n=4$, as computed by Kunyavski\u{\i} in \cite{kun2}. 

The authors of \cite{hoshi} also pay special attention to the case $n=5$ (see Example 8.1 of \cite{hoshi}). In this case, they establish that the torus $T=R^{1}_{K/k} \mathbb{G}_m$ is stably $k$-rational (see Corollary 1.11 of \cite{hoshi}), i.e. that there exists $n \in \mathbb{N}$ such that $T \times_{k} \mathbb{G}_{m}^n$ is $k$-rational. In the language of group modules, this is equivalent to any flasque module $M$ in a flasque resolution of $\hat{T}$ being a permutation module, which is a stronger property than being coflasque.

\end{rem}

\medskip

The computational method developed in this section might be of independent interest, as it can often be used to compute the birational invariant $\operatorname{H}^1(G,M)$ for low-degree field extensions and, in this way, deduce consequences about the groups $A(T)$ and $\mathfrak{K}(K/k)$.

\pagebreak
\section*{Appendix}

\begin{rem}

The code for all the functions below can also be found at {\url{https://sites.google.com/view/andre-macedo/code}}. Additionally, in order to successfully run the function \code{FlasqCoho}, the user will need the GAP programs for the functions $\code{ConjugacyClassesSubgroups2}$, $\code{H1}$ and $\code{FlabbyResolution}$ (see Sections 4.1, 5.0 and 5.1 of \cite{hoshi}, respectively).
\end{rem}

\begin{lstlisting}
row:=function(s,d)
    local r,k;

    r:=[]; &\Comment{\textcolor{gray}{// i-th row of $R_g$}}&
	
    if s = d then &\Comment{\textcolor{gray}{// Case 2 of page 9}}&
	    r:=List([1..d-1],x->-1);
    else &\Comment{\textcolor{gray}{// Case 1 of page 9}}&
	    for k in [1..d-1] do 
		    if k = s then
			    r:=Concatenation(r,[1]);
		    else
			    r:=Concatenation(r,[0]);
	    	fi;
	    od;
    fi;
    return r; 
end;
\end{lstlisting}
\medskip

\begin{lstlisting}
action:=function(G,H)
    local d,gens,RT,L,S,j,Rg,i,s;
    	
    d:=Order(G)/Order(H);
    gens:=GeneratorsOfGroup(G);
    RT:=RightTransversal(G,H); 
    L:=List(RT,i->CanonicalRightCosetElement(H,i));
    &\Commenttt{\textcolor{gray}{// List of right-coset representatives of $H$ in $G$}}&
    S:=[]; &\Commentt{\textcolor{gray}{// List of matrices $R_g$ for $g \in \code{gens}$}}&
    
    for j in [1..Size(gens)] do
        Rg:=List([1..d-1],x->0); &\Commentt{\textcolor{gray}{// Create a matrix with $d-1$ lines}}&
       	for i in [1..d-1] do
       	    s:=PositionCanonical(RT,L[i]*gens[j]);
       	    &\Commenttt{\textcolor{gray}{// Obtain the index s $ =\sigma(i)$ of the right-coset $(H * L[i]).$\code{gens}[j]  in $RT$ }}&
       	    &\Commenttt{\textcolor{gray}{as explained on page 9}}&
            Rg[i]:=row(s,d); &\Commentt{\textcolor{gray}{// Produce the $i$-th row of $R_{\code{gens}[j]}$}}&
        od;
        S:=Concatenation(S,[Rg]); &\Commentt{\textcolor{gray}{// Append the matrix $R_{\code{gens}[j]}$ to $S$}}&
    od;
    return GroupByGenerators(S);  &\Commentt{\textcolor{gray}{// Return the group $R_G$}}&
end;
\end{lstlisting}

\medskip

\begin{lstlisting}
FlasqCoho:=function(G,H)
    local RG,FR,FM;

    RG:=action(G,H); &\Commentt{\textcolor{gray}{// Matrix group $R_G$}}&
    FR:=FlabbyResolution(RG); &\Commentt{\textcolor{gray}{// Flasque resolution of $R_G$}}&
    FM:=FR.actionF; &\Commentt{\textcolor{gray}{// Flasque module in \code{FR}}}&
        
    return H1(FM); &\Commentt{\textcolor{gray}{// Return the cohomology group $\operatorname{H}^1(G,M)$}}&
end;

\end{lstlisting}

\end{document}